\documentclass{article}
\usepackage{harvard}
\usepackage{graphicx}
\usepackage{camnum}
\usepackage{amsmath}
\usepackage{amsfonts}
\usepackage{arydshln}

\usepackage{dutchcal}
\usepackage{listings}
\usepackage{mathalfa}
\usepackage{mathrsfs}

\usepackage{bm} 
\usepackage{fixmath} 

\usepackage{amsfonts}
\usepackage{graphicx}
\usepackage{epstopdf}
\usepackage{algorithmic}

\numberwithin{figure}{section} \numberwithin{equation}{section}

\allowdisplaybreaks

\newcommand{\G}{\CC{\Gamma}}
\newcommand{\DDD}{\mathcal{D}}
\newcommand{\ee}{{\mathrm e}}

\newcommand{\PP}{\mathrm{P}}
\newcommand{\LL}{\mathrm{L}}
\newcommand{\hyper}[5]
       {{}_{#1} F_{#2} \!\left[
           \begin{array}{l}
         #3;\\#4;
           \end{array}#5\right]
       }

\begin{document}
\title{A recurrence relation for generalised connection coefficients}
\author{Jing Gao and Arieh Iserles}

\maketitle
\tableofcontents

\thispagestyle{empty}

\begin{abstract}
  We formulate and prove a general recurrence relation that applies to integrals involving orthogonal polynomials and similar functions. A special case are connection coefficients between two sets of orthonormal polynomials, another example is integrals of products of Legendre functions.
\end{abstract}

\section{A universal recurrence for connection coefficients}

Let two Borel measures, $\D\mu$ and $\D\nu$, be given. We denote by $\mathcal{P}=\{p_n\}_{n\in\mathbb{Z}_+}$ and $\mathcal{Q}=\{q_n\}_{n\in\mathbb{Z}_+}$ respectively the orthonormal polynomials with respect to $\D\mu$ and $\D\nu$. It is elementary that both $\mathcal{P}$ and $\mathcal{Q}$ obey three-term recurrence relations of the form
\begin{Eqnarray}
  \label{eq:1.1}
  b_n p_{n+1}(x)&=&(x-a_n)p_n(x)-b_{n-1}p_{n-1}(x),\qquad n\in\mathbb{Z}_+,\\
  \label{eq:1.2}
  d_n q_{n+1}(x)&=&(x-c_n)q_n(x)-d_{n-1}q_{n-1}(x),
\end{Eqnarray}
\cite{chihara78iop}, where $b_{-1}=d_{-1}=0$ and $b_n, d_n>0$ for $n\in\mathbb{Z}_+$. The purpose of this short paper is to investigate the elements of the infinite matrix $\DDD$, where
\begin{equation}
  \label{eq:1.3}
  \DDD_{m,n}=\int_{-\infty}^\infty p_m(x)q_n(x)\D\eta(x),\qquad m,n\in\mathbb{Z}_+,
\end{equation}
where $\D\eta$ is yet another, possibly signed, Borel measure -- note that $\D\mu$, $\D\nu$ and $\D\eta$ need not be distinct. Indeed, the case $\D\nu=\D\eta$ is classical and corresponds to {\em connection coefficients\/} \cite[p.~253]{Ismail}, expressing one set of orthonormal polynomials in terms of another.

Expressions of the form \R{eq:1.3} feature in \cite{Arieh23} and in forthcoming publications of the current authors. It is our contention in this paper that they all obey a very helpful recurrence relation. While both the recurrence and its proof are elementary, they are (to the best of our knowledge) new and they have significant ramifications, some of  which are described in the sequel.

\begin{lemma}
  It is true that
  \begin{equation}
  \label{eq:1.4}
    b_{m-1}\DDD_{m-1,n}+a_m\DDD_{m,n}+b_m\DDD_{m+1,n}=d_{n-1}\DDD_{m,n-1}+c_n\DDD_{m,n}+d_n\DDD_{m,n+1}
  \end{equation}
  for all $m,n\in\mathbb{Z}_+$.
\end{lemma}

\begin{proof}
  It follows from \R{eq:1.1} and $b_m>0$ that
  \begin{displaymath}
    p_{m+1}(x)=\frac{1}{b_m} [(x-a_m)p_m(x)-b_{m-1}p_{m-1}(x)]
  \end{displaymath}
  and substitution in \R{eq:1.3} results in
  \begin{Eqnarray*}
    \DDD_{m+1,n}&=&\frac{1}{b_m}\int_{-\infty}^\infty [(x-a_m)p_m (x)-b_{m-1}p_{m-1}(x)]q_n(x)\D\eta(x)\\
    &=&\frac{1}{b_m} \int_{-\infty}^\infty x p_m(x)q_n(x)\D\eta(x)-\frac{a_m}{b_m} \DDD_{m,n}-\frac{b_{m-1}}{b_m}\DDD_{m-1,n}.
  \end{Eqnarray*}
  However, it is a consequence of \R{eq:1.2} that
  \begin{displaymath}
    xq_n(x)=d_n q_{n+1}(x)+c_n q_n(x)+d_{n-1}q_{n-1}(x)
  \end{displaymath}
  and substitution within the integral yields \R{eq:1.4} after elementary algebra.
\end{proof}

Let us ponder briefly the meaning of \R{eq:1.4}. It is essentially a `cross rule',
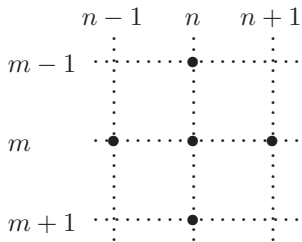
\begin{displaymath}
  \begin{picture}(60,80)
    \put (-40,-4) {$m+1$}
    \put (-40,26) {$m$}
    \put (-40,56) {$m-1$}
    \multiput (-8,-0.5)(4,0){20} {.}
    \multiput (-8,29.5)(4,0){20} {.}
    \multiput (-8,59.5)(4,0){20} {.}
    \put (-12,74) {$n-1$}
    \put (27,74) {$n$}
     \put (48,74) {$n+1$}
    \multiput (-1,68)(0,-4){20} {.}
    \multiput (29,68)(0,-4){20} {.}
    \multiput (59,68)(0,-4){20} {.}
    \put (0,30) {\circle*{4}}
    \put (30,0) {\circle*{4}}
    \put (30,30) {\circle*{4}}
    \put (30,60) {\circle*{4}}
    \put (60,30) {\circle*{4}}
  \end{picture}
\end{displaymath}
Once the boundary conditions $\DDD_{m,0}$ and $\DDD_{0,n}$, are provided and because $\DDD_{m,-1}=\DDD_{-1,n}=0$,\footnote{Since $p_{-1},q_{-1}\equiv0$.} they allow us in tandem with \R{eq:1.4} to fill in the remainder of $\DDD$ in a recursive manner.

Our narrative is focussed on orthonormal polynomials but it can be easily generalised to other normalisations, e.g.\ to monic orthogonal polynomials. In Section~3 we apply similar construction to Legendre functions, which in general are not polynomial.

\section{Few examples}

\subsection{A pair of Legendre weights}

The trivial case $\D\mu=\D\nu=\D\eta$ results in a unit matrix, as it of course should. As a somewhat less elementary example let us consider $\D\mu(x)=\D\nu(x)=\chi_{(-1,1)}(x)\D x$ and $\D\eta(x)=\chi_{(-1,1)}(x)x^2\D x$. Thus, $p_n=q_n=\sqrt{n+\frac12}\CC{P}_n$, normalised standard Legendre polynomials, and one can prove easily from the standard three-term recurrence for Legendre polynomials that
\begin{displaymath}
  a_n,c_n\equiv0,\qquad b_n=d_n=\frac{n+1}{\sqrt{(2n+1)(2n+3)}}.
\end{displaymath}
Therefore it follows from \R{eq:1.4} that
\begin{Eqnarray*}
  &&\frac{m}{\sqrt{(2m-1)(2m+1)}}\DDD_{m-1,n}+\frac{m+1}{\sqrt{(2m+1)(2m+3)}}\DDD_{m+1,n}\\
  &=&\frac{n}{\sqrt{(2n-1)(2n+1)}}\DDD_{m,n-1}+\frac{n+1}{\sqrt{(2n+1)(2n+3)}}\DDD_{m,n+1}.
\end{Eqnarray*}
Finally, $\DDD$ is symmetric, $\DDD_{0,n}=0$ for $n\geq3$ (because $x^2 p_0$ is then orthogonal to $p_n$) and
\begin{displaymath}
  \DDD_{0,0}=\frac13, \qquad \DDD_{0,1}=0,\qquad \DDD_{0,2}=\frac{2\sqrt{5}}{15}.
\end{displaymath}
In general, $\DDD_{m,n}=0$ for $m+n$ odd and for $|m-n|\geq3$, while
\begin{displaymath}
  \DDD_{n,n}=\frac{2n^2+2n-1}{(2n-1)(2n+3)},\qquad \DDD_{n+2,n}=\DDD_{n,n+2}=\frac{(n+1)(n+2)}{(2n+3)\sqrt{(2n+1)(2n+5)}}.
\end{displaymath}
This can be verified at once from the above recurrence relation.

\subsection{Connection coefficients}

$\mathcal{Q}$ is a basis of the space of polynomials , hence each $p_m$ can be expanded as a linear combination of $q_0,q_1,\ldots,q_m$: the coefficients are $\DDD_{m,n}$ with $\D\eta=\D\nu$. These {\em connection coefficients\/} \cite[p.~255--261]{Ismail} are of importance in the theory of orthogonal polynomials and they are known explicitly in some situations, e.g.
\begin{displaymath}
  \PP_m^{(\gamma,\delta)}(x)=\sum_{n=0}^m \DDD_{m,n} \PP_n^{(\alpha,\beta)}(x),\qquad \alpha,\beta,\gamma,\delta>-1,
\end{displaymath}
where
\begin{Eqnarray*}
  \DDD_{m,n}&=&\frac{(\gamma+n+1)_{m-n}(m+\gamma+\delta+1)_n}{(m-n)!\G(\alpha+\beta+2n+1)} \G(\alpha+\beta+n+1)\\
  &&\mbox{}\times \hyper{3}{2}{-m+n,m+n+\gamma+\delta+1,\alpha+n+1}{\gamma+n+1,\alpha+\beta+2n+2}{1}
\end{Eqnarray*}
\cite[Thm.~9.1.1]{Ismail}.\footnote{Jacobi polynomials, of course, are not orthonormal but the formula can be easily transformed to our setting.}

The matrix $\DDD$ of connection coefficients is lower triangular, invertible and $\DDD^{-1}$ yields the connection coefficients expressing the $q_m$s in terms of $p_0,p_1,\ldots,p_m$. As an example (which we did not find in literature) let us `connect' two kinds of Laguerre polynomials: express $\LL_m^{(\alpha)}$ in terms of $\LL_n^{(\beta)}$ for $n\leq m$, where $\alpha,\beta>0$. We have
\begin{displaymath}
  \LL_{n+1}^{(\beta)}(x)=\left(-\frac{x}{n+1}+\frac{2n+\beta+1}{n+1}\right)\!\LL_n^{(\beta)}(x)
  -\frac{n+\beta}{n+1}\LL_{n-1}^{(\beta)}(x),
\end{displaymath}
while the orthonormal Laguerre polynomials (with a positive multiple of $x^n$) are
\begin{displaymath}
  \tilde{\LL}_n^{(\beta)}(x)=(-1)^n \sqrt{\frac{n!}{\G(n+\beta+1)}} \LL_n^{(\beta)}(x)
\end{displaymath}
and they obey  the three-term recurrence relation with
\begin{displaymath}
  b_n=\sqrt{(n+1)(n+1+\beta)},\quad a_n=2n+1+\beta,\qquad n\in\mathbb{Z}_+.
\end{displaymath}
The recurrence \R{eq:1.4} yields
\begin{Eqnarray}
  \nonumber
  &&\sqrt{m(m+\alpha)}\DDD_{m-1,n}+(2m+1+\alpha)\DDD_{m,n}+\sqrt{(m+1)(m+1+\alpha)}\DDD_{m+1,n}\\
  \label{eq:2.1}
  &=&\sqrt{n(n+\beta)}\DDD_{m,n-1}+(2n+1+\beta)\DDD_{m,n}+\sqrt{(n+1)(n+1+\beta)}\DDD_{m,n+1},\hspace*{25pt}
\end{Eqnarray}
where
\begin{displaymath}
\DDD_{m,n}=\int_0^\infty \tilde{L}_m^{(\alpha)}(x) \tilde{L}_n^{(\beta)}(x) x^\beta \ee^{-x}\D x.
\end{displaymath}

This need be accompanied by boundary conditions. It is clear that $\DDD_{0,n}\equiv0$ for $n\in\mathbb{N}$ (recall that $\DDD$ is lower triangular!) while, since $\tilde{\LL}_0^{(\beta)}\equiv 1/\sqrt{\G(1+\beta)}$, we have
\begin{Eqnarray*}
  \DDD_{m,0}&=&\frac{1}{\sqrt{\G(1+\beta)}} \int_0^\infty \tilde{L}_m^{(\alpha)}(x) x^\beta \ee^{-x}\D x\\
  &=&(-1)^m \sqrt{\frac{m!}{\G(m+\alpha+1)\G(1+\beta)}} \int_0^\infty \LL_m^{(\alpha)}(x) x^\beta \ee^{-x}\D x.
\end{Eqnarray*}

To derive $\frak{d}_m=\int_0^\infty \LL_m^{(\alpha)}(x) x^\beta \ee^{-x}\D x$ we use the standard recurrence relation for Laguerre polynomials \cite[18.12.13]{dlmf},
\begin{displaymath}
  \sum_{m=0}^\infty \LL_m^{(\alpha)}(x)z^m=\frac{1}{(1-z)^{1+\alpha}} \exp\!\left(-\frac{xz}{1-z}\right)\!,\qquad 0<z<1.
\end{displaymath}
Therefore
\begin{displaymath}
  \sum_{m=0}^\infty \frak{d}_m z^m=\frac{1}{(1-z)^{1+\alpha}} \int_0^\infty \exp\!\left(-\frac{x}{1-z}\right)\! x^\beta\D x=\G(1+\beta)(1-z)^{\beta-\alpha} .
\end{displaymath}
Expanding the binomial function we obtain
\begin{displaymath}
  \DDD_{m,0}=(-1)^m \sqrt{\frac{\G(1+\beta)}{m!\G(m+1+\alpha)}} (\alpha-\beta)_m,\qquad m\in\mathbb{Z}_+.
\end{displaymath}
Substitution in \R{eq:2.1} conforms that the explicit formula for Laguerre--Laguerre connection coefficients is
\begin{displaymath}
  \DDD_{m,n}=(-1)^{m-n} \frac{(\alpha-\beta)_{m-n}}{(m-n)!} \sqrt{\frac{m!\G(n+1+\beta)}{n!\G(m+1+\alpha)}} ,\qquad 0\leq n\leq m.
\end{displaymath}

\subsection{Matrices $\mathcal{D}$ that occur in the analysis of orthonormal systems}

The recurrence \R{eq:2.1} is independent of the choice of $\D\eta$. A particular choice of a {\em signed\/} measure $\D\eta$, together with the choice $\alpha=\beta$, features in \cite{Arieh23} in the analysis of certain systems orthonormal in $\LL_2(a,b)$. It has been ito a large extent the original motivation to consider expressions of the form \R{eq:1.1}. Specifically, requiring $\alpha>0$, the functions
\begin{displaymath}
  \varphi_n(x)=\sqrt{\frac{n!}{\G(n+1+\alpha)}} x^{\alpha/2}\ee^{-x/2} \LL_n^{(\alpha)}(x),\qquad n\in\mathbb{Z}_+,
\end{displaymath}
form an orthonormal set in $\CC{L}_2[0,\infty)$ and the {\em differentiation matrix\/}
\begin{displaymath}
  \mathcal{E}_{m,n}=\int_0^\infty \frac{\D \varphi_m(x)}{\D x} \varphi_n(x)\D x,\qquad m,n\in\mathbb{Z}_+,
\end{displaymath}
is skew symmetric. The evaluation of the entries of $\mathcal{E}$ led in \cite{Arieh23} to the evaluation of expressions of the form
\begin{displaymath}
  \int_0^\infty \tilde{\LL}_m^{(\alpha)}(x)\tilde{\LL}_n^{(\alpha)}(x) \frac{\D x^\alpha \ee^{-x}}{\D x}\D x,\qquad m,n\in\mathbb{Z}_+,
\end{displaymath}
hence \R{eq:1.1} with Laguerre weights $\D\mu(x)=\D\nu(x)=\chi_{(0,\infty)}(x) x^\alpha \ee^{-x}\D x$ and the {\em signed\/} measure
\begin{displaymath}
  \D\eta(x)=\frac{\D x^\alpha \ee^{-x}}{\D x}\D x=(\alpha-x)x^{\alpha-1}\D x,\qquad x>0.
\end{displaymath}
(As a matter of fact, because of orthogonality, we might replace this by $\D\eta(x)=\alpha x^{\alpha-1}\D x$.) The recurrence \R{eq:2.1} remains valid (with $\beta=\alpha$). The $\mathcal{E}_{m,n}$s have been derived explicitly in \cite{Arieh23} by long algebra -- instead we can verify in a straightforward manner that
\begin{displaymath}
  \mathcal{E}_{m,n}=-\frac12 \sqrt{\frac{m!\G(n+1+\alpha)}{\G(m+1+\alpha)n!}},\qquad m \geq n+1,
\end{displaymath}
with symmetric completion for $m < n$, is consistent with \R{eq:2.1} and obeys the boundary conditions for $n=0$ and $m=0$.  In addition, $\mathcal{E}_{m,m}= 0$.

Similar expression, again originating in \cite{Arieh23}, is
\begin{displaymath}
  \mathcal{F}_{m,n}=\int_{-1}^1 (1-x^2)^{\alpha-1} \PP_m^{(\alpha,\alpha)}(x)\PP_n^{(\alpha,\alpha)}(x)\D x,\qquad m,n\in\mathbb{Z}_+,
\end{displaymath}
where $\alpha>0$ and $\PP_n^{(\alpha,\alpha)}$ is an ultraspherical polynomial. Thus, $\D\mu(x)=\D\nu(x)=\chi_{(-1,1)}(x)(1-x^2)^\alpha\D x$ and $\D\eta(x)=\chi_{(-1,1)}(x)(1-x^2)^{\alpha-1}\D x$. It is enough to consider the case $m+n$ even, otherwise $\mathcal{F}_{m,n}$ vanishes. Using exactly the same approach as for Laguerre weights above (or a much longer and more convoluted algebra in \cite{Arieh23}) we can prove that
\begin{displaymath}
  \mathcal{F}_{m,n}=\frac{4^\alpha}{\alpha} \frac{\G(m+1+\alpha)\G(n+1+\alpha)}{n!\G(m+1+2\alpha)},\qquad m\geq n,\quad m+n\mbox{\ even},
\end{displaymath}
with symmetric completion for $m\geq n$.

\section{Associated Legendre functions}

The solutions of the differential equation
\begin{displaymath}
  (1-x^2)y''-2xy'+\left[\nu(\nu+1)-\frac{\mu^2}{1-x^2}\right]\!y=0,\qquad x\in(-1,1),
\end{displaymath}
where $\mu,\nu\in\mathbb{C}$, are {\em  Legendre functions\/} $\PP_\nu^\mu$ and {\em associated Legendre functions\/} $\mathrm{Q}_\nu^\mu$. These functions feature in a wide range of applications, not least as a main component of {\em spherical harmonics.\/} They admit a number of confusingly different normalisations \cite{abramowitz66hma,arfken66mmp,courant53mmp,dlmf} yet, in this paper, being interested in $\nu\in\mathbb{Z}_+$ and $\mu\in\{0,1,\ldots,\nu\}$ and restricting our attention to the Legendre functions $\PP_n^m$, we use the expression
\begin{equation}
  \label{eq:3.1}
  \PP_n^m(x)=\frac{(-n)_m(1+n)_m}{m!} \left(\frac{1-x}{1+x}\right)^{\!m/2}\hyper{2}{1}{-n,n+1}{m+1}{\frac{1-x}{2}}\!,
\end{equation}
where $(z)_k=z(z+1)\cdots (z+k-1)$ is the familiar {\em Pochhammer symbol\/} and ${}_2F_1$ is a {\em hypergeometric function.\/}

Note that the $\PP_n^m$s are, in general, non-polynomial. Thus, while $\PP_n^0=\PP_n$, the standard Legendre polynomial, $\PP_n^1(x)=-\sqrt{1-x^2}\PP_n'(x)$ -- cf.\ \cite[14.6.1]{dlmf} for explicit expressions for $\PP_n^m$.

In our forthcoming work on multivariate expansions we have encountered the problem to calculate expressions of the form
\begin{equation}
  \label{eq:3.2}
  g^m_{\ell,n}=g_{n,\ell}^m=\int_{-1}^1 \frac{\PP_\ell^m(x)\PP_n^m(x)}{\sqrt{1-x^2}}\D x,\qquad m\leq\min\{\ell,n\}.
\end{equation}
Since it follows readily from \R{eq:3.1} that $\PP_\ell^m(-x)=(-1)^{m+\ell}\PP_\ell^m(x)$, we deduce that $g^m_{\ell,n}=0$ when $\ell+n$ is odd.

Our starting point is the three-term recurrence relation
\begin{equation}
  \label{eq:3.3}
  (n-m+1)\PP_{n+1}^m(x)-(2n+1)x\PP_n^m(x)+(m+n)\PP_{n-1}^m(x)=0,
\end{equation}
which can be easily derived from \cite[14.10.3]{dlmf}.

We proceed similarly to Section~1. Combining \R{eq:3.2} and \R{eq:3.3}, we have
\begin{Eqnarray*}
  &&\int_{-1}^1\frac{\PP^m_\ell(x)}{\sqrt{1-x^2}}\! \left[\frac{n-m+1}{2n+1}\PP_{n+1}^m(x)+\frac{m+n}{2n+1}\PP_{n-1}^m(x)\right]\!\D x=\int_{-1}^1 \frac{x\PP_\ell^m(x)\PP_n^m(x)}{\sqrt{1-x^2}}\D x\\
  &=&\int_{-1}^1 \!\left[\frac{\ell-m+1}{2\ell+1} \PP_{\ell+1}^m(x)+\frac{m+\ell}{2\ell+1} \PP_{\ell-1}^m(x)\right]\!\PP_n^m(x)\D x
\end{Eqnarray*}
and  deduce the recurrence
\begin{equation}
  \label{eq:3.4}
  \frac{n-m+1}{2n+1} g_{n+1,\ell}^m+\frac{m+n}{2n+1} g_{n-1,\ell}^m=\frac{\ell-m+1}{2\ell+1}g_{n,\ell+1}^m+\frac{m+\ell}{2\ell+1} g_{n,\ell-1}^m,
\end{equation}
valid for all $m\leq \min\{\ell,n\}$. Once $n+\ell$ is even, all the terms vanish, but not so for an odd $n+\ell$.

For completeness, the recurrence \R{eq:3.4} requires boundary conditions. Given $m\in\mathbb{Z}_+$, we need $g_{n,m}^m=g_{m,n}^m$, $n\geq m$. In addition we need $g_{-1,m}^m+g_{m,1}^m=0$, $m\in\mathbb{N}$: this follows at once from \R{eq:3.1}.

It is an immediate consequence of \R{eq:3.1} that
\begin{displaymath}
  \PP_m^m(x)=(-1)^m \frac{(2m)!}{2^mm!} (1-x^2)^{m/2},\qquad m\in\mathbb{Z}_+,
\end{displaymath}
therefore
\begin{displaymath}
  g_{m,n}^m=(-1)^m \frac{(2m)!}{2^mm!}  \int_{-1}^1 (1-x^2)^{(m-1)/2} \PP_n^m(x)\D x,\qquad n\geq m.
\end{displaymath}
Of course, our interest is restricted to even values of $m+n$, otherwise $g_{m,n}^m$ vanishes.

Let
\begin{equation}
  \label{eq:3.5}
  G_n^m=\int_{-1}^1 (1-x^2)^{(m-1)/2}\PP_n^m(x)\D x,\qquad H_n^m=\int_{-1}^1 x(1-x^2)^{(m-1)/2}\PP_n^m(x)\D x
\end{equation}
for $0\leq m\leq n$, $m+n$ even. Thus, $g_{m,n}^m=(-1)^m (2m!)/(2^mm!) G_n^m$. It follows from \R{eq:3.3}, though, that
\begin{displaymath}
  (n-m+1)G^m_{n+1}+(m+n)G^m_{n-1}=(2n+1)H_n^m.
\end{displaymath}
Because of \cite[14.10.5]{dlmf} it is true that
\begin{displaymath}
  (1-x^2)\frac{\D\PP_n^m(x)}{\D x}=(m+n)\PP_{n-1}^m(x)-nx\PP_n^m(x)
\end{displaymath}
and it follows that
\begin{displaymath}
  H_n^m=\frac{1}{m+1}[(m+n) G_{n-1}^m -n H_n^m]
\end{displaymath}
-- we deduce that
\begin{displaymath}
  H_n^m=\frac{m+n}{m+n+1} G_{n-1}^m.
\end{displaymath}
It now follows from \R{eq:3.5} that
\begin{displaymath}
  G^m_{n+1}=\frac{n^2-m^2}{(n+1)^2-m^2}G^m_{n-1}
\end{displaymath}
and
\begin{displaymath}
  g_{m,n+1}^m=\frac{n^2-m^2}{(n+1)^2-m^2} g_{m,n-1}^m,\qquad m\leq n-1.
\end{displaymath}
(Of course, all this makes sense only for an odd $m+n$.). To start the recursion we need $g_{m,m}^m$ and to this end we note from \R{eq:3.1} that
\begin{displaymath}
  \PP_m^m(x)=(-1)^m \frac{(2m)!}{2^mm!} (1-x^2)^{m/2}
\end{displaymath}
and direct integration yields
\begin{displaymath}
  g_{m,m}^m=\pi (2m)! \!\left[\frac{1}{4^m}{{2m}\choose m}\right]^{\!2},\qquad m\in\mathbb{Z}_+.
\end{displaymath}
It is now a straightforward (yet tedious calculation) that
\begin{equation}
  \label{eq:3.6}
  g^m_{m,m+2k}=g^m_{m+2k,m}=\pi\frac{(2m)!}{4^{2m+2k}} {{2k}\choose k}{{2m}\choose m}{{2m+2k}\choose{m+k}},\qquad k\in\mathbb{Z}_+.
\end{equation}

Using the recurrence \R{eq:3.4} we derive explicitly, by long yet straightforward algebra,
\begin{equation}
  \label{eq:3.7}
  g_{m+1,m+2k+1}^m=g_{m+2k+1,m+1}=\frac{2\pi(m+1)}{4^{2m+2k+1}} (2m)!{{2k}\choose k}{{2m+1}\choose m}{{2m+2k+1}\choose{m+k}}
\end{equation}
for all $m,k\in\mathbb{Z}_+$.

Remaining values of $g_{n,\ell}^m$, $m\leq \min\{n,\ell\}$, can be filled in using the recurrence \R{eq:3.4}. Bearing in mind formul{\ae} \R{eq:3.6} and \R{eq:3.7}, we may expect that the remaining values of $g^m_{n,\ell}$ would be similarly neat and regular. Indeed,
\begin{displaymath}
  g_{m+2,m}^m=g_{m,m+2}^m=\frac{\pi(2m)!}{4^{2m+1}} {{2m}\choose m}{{2m+1}\choose m},\qquad m\in\mathbb{Z}_+,
\end{displaymath}
vindicating this expectation. However,
\begin{displaymath}
  g^m_{m+2,m+2}=\frac{\pi(2m+1)!}{2\cdot 4^{2m+2}} \frac{(8m^2+20m+11)}{2m+3} {{2m}\choose m}{{2m+3}\choose{m+1}}\qquad m\in\mathbb{Z}_+,
\end{displaymath}
and expressions do not become nicer for larger values of $k$. Yet, for practical uses in approximation algorithms, all we need is numerical values of the $g_{n,\ell}^m$s, and this can be easily produced by the recurrence \R{eq:3.4}.

\section{Conclusions}

Our point of departure, Lemma~1, is an exceedingly simple result with trivial proof. Yet, its ramifications are far and wide, since it allows for simple evaluation of connection coefficients and their generalisations and applies whenever a three-term recurrence relation is available for a set of functions: this is obviously true for orthogonal polynomials but also, for example, for Legendre functions $\PP_n^m$ which in general are neither polynomial nor orthogonal.

There is no free lunch: to use the recurrence \R{eq:1.4} or its generalisations we need first to determine boundary conditions. Yet, this is typically much easier than computing $\DDD_{m,n}$ for all $m,n\in\mathbb{Z}_+$ and it often suffices to do so numerically.

\bibliographystyle{agsm}
\bibliography{refs}

@book {abramowitz66hma,
     TITLE = {Handbook of mathematical functions, with formulas, graphs, and
              mathematical tables},
    EDITOR = {Abramowitz, Milton and Stegun, Irene A.},
 PUBLISHER = {Dover Publications, Inc., New York},
      YEAR = {1966},
     PAGES = {xiv+1046},
   MRCLASS = {65.05 (00.20)},
  MRNUMBER = {208797},
}

@book {arfken66mmp,
    AUTHOR = {Arfken, George},
     TITLE = {Mathematical methods for physicists},
 PUBLISHER = {Academic Press, New York-London},
      YEAR = {1966},
     PAGES = {xvi+654},
   MRCLASS = {69.00},
  MRNUMBER = {205512},
MRREVIEWER = {C.-B.\ Ling},
}

@book {chihara78iop,
    AUTHOR = {Chihara, T. S.},
     TITLE = {An Introduction to Orthogonal Polynomials},
    SERIES = {Mathematics and its Applications, Vol. 13},
 PUBLISHER = {Gordon and Breach Science Publishers, New York-London-Paris},
      YEAR = {1978},
     PAGES = {xii+249},
      ISBN = {0-677-04150-0},
   MRCLASS = {42A52},
  MRNUMBER = {0481884},
MRREVIEWER = {A. G. Law},
}

@book {courant53mmp,
    AUTHOR = {Courant, R. and Hilbert, D.},
     TITLE = {Methods of mathematical physics. {V}ol. {I}},
 PUBLISHER = {Interscience Publishers, Inc., New York, N.Y.},
      YEAR = {1953},
     PAGES = {xv+561},
   MRCLASS = {79.0X},
  MRNUMBER = {65391},
MRREVIEWER = {J.\ B.\ Diaz},
}

@techreport {Arieh23,
  AUTHOR = {Iserles, Arieh},
  TITLE = {Orthogonal systems for time-dependent spectral methods},
  YEAR = {2023},
  INSTITUTION = {https://arxiv.org/abs/2302.04217},
}

@book {dlmf, 
     TITLE = {N{IST} {H}andbook of {M}athematical {F}unctions},
    EDITOR = {Olver, Frank W. J. and Lozier, Daniel W. and Boisvert, Ronald
              F. and Clark, Charles W.},
      NOTE = {With 1 CD-ROM (Windows, Macintosh and UNIX)},
 PUBLISHER = {U.S. Department of Commerce, National Institute of Standards
              and Technology, Washington, DC; Cambridge University Press,
              Cambridge},
      YEAR = {2010},
     PAGES = {xvi+951},
      ISBN = {978-0-521-14063-8},
   MRCLASS = {33-00 (00A20 65-00)},
  MRNUMBER = {2723248},
}

@book {Ismail,
    AUTHOR = {Ismail, Mourad E. H.},
     TITLE = {Classical and quantum orthogonal polynomials in one variable},
    SERIES = {Encyclopedia of Mathematics and its Applications},
    VOLUME = {98},
      NOTE = {With two chapters by Walter Van Assche,
              With a foreword by Richard A. Askey},
 PUBLISHER = {Cambridge University Press, Cambridge},
      YEAR = {2005},
     PAGES = {xviii+706},
      ISBN = {978-0-521-78201-2; 0-521-78201-5},
   MRCLASS = {33-02 (05A30 05E35 33C45 33D50 42C05)},
  MRNUMBER = {2191786},
MRREVIEWER = {Bruce C. Berndt},
       DOI = {10.1017/CBO9781107325982},
       URL = {https://doi.org/10.1017/CBO9781107325982},
}

\end{document}